\documentclass[12pt, reqno]{amsart}
\usepackage{mathrsfs}
\usepackage{amsmath, amsthm, amscd, amsfonts, amssymb, graphicx }
\usepackage[bookmarksnumbered, colorlinks, plainpages]{hyperref}
\usepackage{bigints}
\hypersetup{colorlinks=true,linkcolor=red, anchorcolor=green, citecolor=cyan, urlcolor=red, filecolor=magenta, pdftoolbar=true}

\newtheorem{theorem}{Theorem}[section]

\newtheorem{corollary}[theorem]{Corollary}
\theoremstyle{definition}

\newtheorem{example}[theorem]{Example}

\theoremstyle{remark}

\numberwithin{equation}{section}

\newcommand{\norm}[1]{\left\Vert#1\right\Vert}

\newcommand{\A}{\mathcal{A}}

\makeatletter
\@namedef{subjclassname@2020}{%
  \textup{2020} Mathematics Subject Classification}
\makeatother
\begin{document}
	
	\setcounter{page}{1}
	\title[$k$-quasi $n$-power posinormal Weighted]{ $k$-quasi $n$-power posinormal  Weighted Composition and Cauchy Dual of Moore-Penrose inverse of Lambert Operators}
	\author[SOPHIYA S DHARAN, T. PRASAD, P. RAMYA, M. H. M. RASHID ]{SOPHIYA S DHARAN, T. PRASAD, P. RAMYA, M. H. M. RASHID}
	\address{Sophiya S Dharan\endgraf Department of Mathematics\endgraf Government Polytechnic College\endgraf Kalamassery, Ernakulam,
		Kerala, India-680567}
	\email{sophysasi@gmail.com}
	\address{T. Prasad \endgraf Department of Mathematics\endgraf University of  Calicut\endgraf Malapuram, Kerala, \endgraf India - 673635}
	\email{prasadt@uoc.ac.in}
	\address{P. Ramya \endgraf Department of Mathematics\endgraf N. S. S College\endgraf Nemmara, Palakkad, Kerala, \endgraf India - 678508}
\email{ramyagcc@gmail.com}
	\address{M.H.M. Rashid \endgraf Department of Mathematics\& Statistics\endgraf Faculty of Science P.O. Box(7)\endgraf Mutah University\endgraf Alkarak-Jordan}
\email{malik\_okasha@yahoo.com}
	\dedicatory{ }
	
	\let\thefootnote\relax\footnote{}

	\subjclass[2020]{47A05, 47B20, 47B37}
	
	\keywords{ posinormal operator, composition operators, weighted composition operators, conditional expectation, cauchy dual, rooted directed trees}
	
\begin{abstract}  In this paper we characterize \(k\)-quasi \(n\)-power posinormal composition operators and weighted composition operators on the Hilbert space \(L^2(\Sigma)\). For Lambert conditional operators (of the form \(T = M_w E M_u\)), we  establish necessary and sufficient conditions under which these Cauchy duals via the Moore-Penrose inverse  become \(k\)-quasi \(n\)-power posinormal operators. Finally, we construct an explicit example of a \(k\)-quasi \(n\)-power posinormal weighted shift operator on a rooted directed tree.
	\end{abstract}\maketitle

\section{Introduction and Preliminaries}

Let $B(\mathscr{H})$ denote the algebra of all bounded linear operators on a complex Hilbert space  $\mathscr{H}$. For an operator $T \in \mathscr{H}$, we denote  $ker(T)$, $T(\mathscr{H})$  and $\sigma(T)$  for the kernel, the range and the spectrum of $T$, respectively. Recall that an  operator $T\in B(\mathcal{H})$ is said to be hyponormal if $T^{*}T \geq TT^{*}$ and posinormal if there exists a positive operator $P \geq 0$ such that $TT^* = T^*PT$ or equivalently, $\lambda^2T^*T - TT^* \geq 0$ for some $\lambda >0$ (see \cite{rha}). Hyponormal operators are posinormal \cite{rha}. Many interesting properties of posinormal operators have been studied recently by many authors (see \cite{Bour1}, \cite{Bour2}, \cite{Bour3}, \cite{Itoh}, \cite{Kubrusly1} and \cite{Kubrusly2}). Recently,  Beiba \cite{Beiba}  introduced and studied  $n$-power posinormal operator. An operator $T$ is said to  be an $n$-power posinormal operator for $n \in \mathbb{N}$ if there exists a positive operator $P \geq 0$ such that $T^nT^{*n} = T^*PT$ or equivalently, $T^nT^{*n} \leq \lambda^2 T^*T$ for some $\lambda > 0$ or $\lambda^2T^*T - T^nT^{*n} \geq 0$ \cite{Beiba}.

An operator  $T \in B(\mathscr{H})$  is  said to be  $k$-quasi $n$-power posinormal  if for $k,n \in \mathbb{N}$ and for some $ \lambda > 0$, $T^{*k}( \lambda^2 T^*T-T^nT^{*n})T^k \geq 0$\cite{sophiya}. In particular if $k=0$, $n=1$, the class of $k$-quasi $n$-power posinormal operators coincides with the class of posinormal operators and if $k=1$, this  coincides with the class of $n$-power posinormal operators.
It is evident that following inclusion holds in general;
  $$\textrm{posinormal}  \subseteq \textrm{$n$-power posinormal } \subseteq \textrm{$k$-quasi $n$-power posinormal.}$$

Let $(\mathcal{X},\Sigma,\mu)$ be a complete $\sigma$-finite measure space. A mapping $T: X \to X $ is said to be a measurable transformation if $T^{-1}(S) \in \Sigma $, for every $S\in \Sigma$. A measurable transformation $T$ is said to be non-singular if $ \mu (T^{-1}(S))=0$, whenever $\mu(S)=0$. If $T$ is non-singular, then the measure
$ \mu \circ T^{-1}$, defined as $ \mu \circ T^{-1}(S)= \mu (T^{-1}(S))$, for every
$S \in \Sigma$, is an absolutely continuous measure on $\Sigma$ with respect to $ \mu $ and is denoted as $ \mu \circ T^{-1} << \mu$. Then by Radon-Nikodym theorem, there exists a non- negative function $h\in L^1( \mu) $ such that  $$\mu \circ T^{-1}(S)= \int_{S} hd \mu ~~~\text{for every } S\in \Sigma.$$ The function $h$ is called the Radon-Nikodym derivative of $ \mu \circ T^{-1}$ with respect to $\mu$, that is, $h=\frac{d \mu\circ T^{-1}}{d \mu}$.

Now, $T$ induces a linear transformation $C_T$ on $L^2 (\mu)$ into the linear space of all measurable functions on $X$, defined as $C_T f=f\circ T$, for every $f\in L^2(\mu)$. If $C_T : L^2(\mu) \to L^2(\mu)$ is a continuous operator,  then it is called a composition operator on $L^2 (\mu)$ induced by $T$. We refer \cite{Nord} and \cite{Singh} for general properties of $C_T$.

The weighted composition operator $W_T$ on $L^2(\mu)$ induced by the non- singular measurable transformation $T$ and an essentially bounded complex-valued measurable function $\pi$, is a linear operator defined by, \\ $$W_Tf=\pi. (f\circ T)= M_{\pi}C_Tf,~~\text{for every}~~f\in L^2(\mu).$$

In case $\pi=1$ a.e., $W_T$ turns to be the composition operator $C_T$.
The measure defined on $\Sigma$ as $\mu_T^{\pi}$ as,
$$\mu_T^{\pi}(S) =\int_{T^{-1}(S)}|\pi|^2d\mu.$$
Then, $\mu_T^{\pi}<<\mu$ and if $h^{\pi}$ denote the Radon-Nikodym derivative of $\mu_T^{\pi}$ with respect to $\mu$. Let $\phi=(h^{\pi})^{\frac{1}{2}}$, then $W_T$ is a bounded operator on $L^2(\mu)$ if and only if $\phi \in L^{\infty}(\mu).$

Let $\mathcal{A}$ be a $\sigma$-subalgebra of $\Sigma$. Then $(\mathcal{X},\mathcal{A},\mu)$ is also a complete $\sigma$-finite measure space. We denote $L^0(\Sigma)$ as the space of complex-valued functions on $\mathcal{X}$ that are measurable with respect to $\Sigma$. The support of a measurable function $f$ is given by $S(f) = \{x \in \mathcal{X}: f(x) \neq 0\}$.

For any non-negative $f \in L^0(\Sigma)$, there exists a measure $\nu_f(B) = \int_B f d\mu$ for every $B \in \mathcal{A}$ that is absolutely continuous with respect to $\mu$. According to the Radon-Nikodym theorem, there exists a unique non-negative $\mathcal{A}$-measurable function $E(f)$ such that
$$\int_BE(f)\,d\mu=\int_Bf\,d\mu\,\,\,\mbox{for all $B\in\A$}.$$

Taking, $\mathcal{A}=T^{-1}(\Sigma)$, we obtain an operator $E$, often referred to as the \\conditional expectation operator associated with $\mathcal{A}$, defined on the set of all non-negative function $f\in L^2(\mu)$, and is uniquely determined by the following conditions:
\begin{enumerate}
\item [(i)]$E(f)$ is $\mathcal {A}$ measurable.
\item [(ii)]$\int_BE(f)\,d\mu=\int_Bf\,d\mu\,\,\,\mbox{for all $B\in \A $}.$
\end{enumerate}

This operator $E$ plays a crucial role in this paper, and we list some of its useful properties here. For $f, g\in L^2(\mu)$, (refer to \cite{Herron, Estaremi2, Rao} for more details):

\begin{enumerate}
  \item [(i)] $E(g)=g$ if and only if $g$ is $\A$ measurable.
  \item [(ii)] If $g$ is $\A$ measurable, then $E(fg)=E(f)g$.
  \item [(iii)] $E(fg\circ T)= (E(f))(g\circ T)$ and $E(E(f)g)=E(f)E(g)$.
  \item [(iv)] $E(1)=1$.
  \item [(v)] If $f\geq 0$, then $E(f)\geq 0$; if $f>0$, then $E(f)> 0$.
  \item [(vi)] $|E(f)|^{p}\leq E(|f|^p)$.
  \item [(vii)] $E$ is the identity operator on $L^2(\mu)$ if and only if $T^{-1}(\Sigma)=\Sigma$.
  \item [(viii)] $E$ is the projection operator from $L^2(\mu)$ onto $\overline{C(L^2(\mu))}$.\\
  \end{enumerate}

 The conditional expectation operator associated with $T^{-n}\Sigma$ is denoted by $E_n$. If $T^{-n}\Sigma$ is purely atomic $\sigma$- algebra of $\Sigma$ generated by the atoms $\{A_k\}_{k>0}$, then
 $$E_n(f|T^{-n}\Sigma)=\sum_{k=0}^{\infty}\frac{1}{\mu(A_k)}(\int_{A_k}fd\mu)\chi_{A_k}.$$

The weighted conditional operator $T_{w,u}: L^2(\Sigma) \to L^2(\Sigma)$ given by $$T_{w,u}(f):=M_wEM_u(f)=wE(uf)$$ is called the Lambert conditional operator. Several properties of these class of operators were well studied in \cite{Estaremi3}, \cite{Estaremi4}, \cite{Jabbar1}.

Let $B_C(\mathscr{H})$ be the set of all bounded linear operators on $\mathscr{H}$ with closed range. For $T\in B_C(\mathscr{H})$, the Moore- Penrose inverse of $T$ is defined as the unique bounded operator $T^{\dagger}$ satisfying,

$$TT^{\dagger}T=T,~T^{\dagger}TT^{\dagger}=T^{\dagger},~(TT^{\dagger})^*=TT^{\dagger},~(T^{\dagger}T)^*=T^{\dagger}T$$

The Moore-Penrose inverse $T^{\dagger}$ is a densely defined  bounded closed operator. If $T$ is invertible, then $T^{-1}=T^{\dagger}$. Refer \cite{Ben}, \cite{Djor} for more important properties of $T^{\dagger}$.

The Cauchy dual of $T\in B_C(\mathscr {H})$ is defined as,
 $ \omega (T)=T(T^*T)^ {-1} $, which is a right inverse of $T^*$. For more aspects and properties of Cauchy dual refer \cite{Anand}, \cite{Chavan2}, \cite{Ezzahraoui}.

This paper investigates the class of \(k\)-quasi \(n\)-power posinormal operators, extending the theory of posinormal operators to broader functional-analytic contexts. The subsequent sections develop characterizations, conditions, and concrete examples for such operators across diverse settings.   In Sections 2 and 3, we characterize $k$-quasi $n$-power posinormal composition and weighted composition operators on the Hilbert space $L^2(\Sigma)$. In Section 4, we examine the Cauchy dual of Lambert conditional operators using the Moore-Penrose inverse and derive a sufficient condition for this Cauchy dual to be a $k$-quasi $n$-power posinormal. Finally, in Section 5, we present an example of a $k$-quasi $n$-power posinormal weighted shift operator defined on a rooted directed tree.

\section{$k$-quasi $n$-power Posinormal Composition Operators}

This section focuses on \(k\)-quasi \(n\)-power posinormal composition operators \(C_T\) on \(L^2(\mu)\). Under the assumption that \(T^{-k}(\Sigma) = \Sigma\) for all \(k\), we establish necessary and sufficient conditions for \(C_T\) and its adjoint \(C_T^*\) to belong to this class. These conditions are expressed in terms of Radon-Nikodym derivatives \(h_k\) and involve inequalities that link the norms of weighted functions. An explicit example on \([0,1]\) with the involution \(T(x) = 1 - x\) illustrates the theory. 

\begin{theorem} \cite{Campbell}\cite{Harrington}
Let $T$ be a measurable function on $X$ and $E$ be the orthogonal projection from $L^2(\mu)$ onto $\overline {R(C_T)}$. Then the following results holds for every $f\in L^2(\mu)$:
\begin{enumerate}
\item[(i)] $C_T^*f = h.E(f)\circ T^{-1} $.
\item[(ii)] $C_T^k f= f\circ T^k$, ${C_T^*}^k f= h_k. E(f)\circ T^{-k}.$
\item[(iii)] $C_T C_T^* f= (h\circ T) Ef$, $C_T^*C_T f=hf$.
\item[(iv)] ${C_T^*}^k C_T ^k f =M_{h_k} f= h_k f$
\end{enumerate}
\end{theorem}

\begin{theorem} \label{THM2.2}

Let $C_T \in B(L^2(\mu ))$. Let $E$ be the projection from $L^2(\mu )$ onto $\overline {R(C_T)}$ and if $T^{-k}(\Sigma)=\Sigma$, for all $k$. Then for $\lambda > 0$ and for all $f \in L^2(\mu )$, the following statements hold:-

\begin{enumerate}
 \item[(i)] $C_T$ is $k$-quasi $n$-power posinormal if and only if,
	 $$\lambda \lVert \sqrt{h_{k+1}}f \rVert \geq \lVert \sqrt{h_kh_n \circ T^{n-k}}f \rVert.$$
	
\item[(ii)]  $C_T^*$ is $k$-quasi $n$-power posinormal if and only if, $$\lambda \lVert \sqrt{h_{k+1} \circ T^{k+1}}f \rVert \geq \lVert \sqrt{h_n h_k \circ T^k}f \rVert.$$
\end{enumerate}	
	
\end{theorem}

\begin{proof}

		(i)  By definition, $C_T$ is $k$-quasi $n$-power posinormal if and only if for every $f\in L^2(\mu)$,
		$$\langle (\lambda ^2{C_T^*}^{k+1} C_T^{k+1}- {C_T^*}^k C_T^n {C_T^*}^n C_T^k ) f, f \rangle ~ \geq 0.$$

Now, $${C_T^*}^{k+1} C_T^{k+1} f =h_{k+1} f.$$\\
 Hence,
 $$\langle \lambda ^2{C_T^*}^{k+1} C_T^{k+1}f, f\rangle = \lambda ^2 \langle h_{k+1}f, f \rangle = \lambda^2 {\lVert \sqrt{h_{k+1}}f\rVert }^2.$$
 		 	
Also,
\begin{align*}
C_T^k {C_T^*}^k f &= C_T^k[h_k. E(f) \circ T^{-k}] \\
&=[h_k. E(f) \circ T^{-k}] \circ T^k \\
&=h_k \circ T^k . E(f) \\
&=h_k \circ T^k. f
\end{align*}

Now,
\begin{align*}
{C_T^*}^k C_T^n {C_T^*}^n C_T^k f &= {C_T^*}^k C_T^n {C_T^*}^n(f\circ T^k) \\
&= {C_T^*}^k[h_n\circ T^n. E(f\circ T^k ]\\
&=h_k. E [h_n\circ T^n.E(f\circ T^k)]\circ T^{-k}\\
&=h_k. [E (h_n\circ T^n)\circ T^{-k}][E(f\circ T^k)\circ T^{-k}]\\
&=h_k.h_n\circ T^{n-k}f
\end{align*}

Therefore, $$\langle {C_T^*}^k C_T^n {C_T^*}^n C_T^k f, f\rangle = \langle h_k.h_n \circ T^{n-k}f, f\rangle = {\lVert \sqrt{h_k.h_n \circ T^{n-k}}f\rVert}^2.$$\\

Now, for every $f\in L^2(\mu)$,

$$\langle (\lambda ^2{C_T^*}^{k+1} C_T^{k+1}- {C_T^*}^k C_T^n {C_T^*}^n C_T^k ) f, f \rangle ~ \geq 0.$$
$$\Longleftrightarrow \lambda^2 {\lVert \sqrt{h_{k+1}}f\rVert }^2 \geq {\lVert \sqrt{h_k.h_n \circ T^{n-k}}f\rVert}^2.$$

Hence the result.\\

(ii) Similarly, $C_T^*$ is $k$-quasi $n$-power posinormal if and only if for every \\$f\in L^2(\mu)$,

		$$\langle (\lambda ^2 C_T^{k+1} {C_T^*}^{k+1}- C_T^k {C_T^*}^n C_T^n {C_T^*}^k) f, f \rangle ~ \geq 0.$$
		Now, $C_T^{k+1} {C_T^*}^{k+1} f= (h_{k+1}\circ T^{k+1}).f  $	\\

Hence,
 $$\langle \lambda ^2C_T^{k+1} {C_T^*}^{k+1}f, f\rangle = \lambda ^2 \langle (h_{k+1}\circ T^{k+1})f, f \rangle = \lambda^2 {\lVert \sqrt{h_{k+1}\circ T^{k+1}}f\rVert }^2.$$

Also,
\begin{align*}
C_T^k {C_T^*}^n C_T^n {C_T^*}^k f &= C_T^k {C_T^*}^n C_T^n[h_k.E(f)\circ T^{-k}] \\
&= C_T^k[h_n.(h_k.E(f)\circ T^{-k} )]\\
&=[h_n.(h_k.E(f)\circ T^{-k} )]\circ T^k\\
&=(h_n\circ T^k).(h_k\circ T^k).E(f)\\
&=(h_n.h_k\circ T^k)f
\end{align*}

Therefore, $$\langle C_T^k {C_T^*}^n C_T^n {C_T^*}^k f, f\rangle = \langle (h_n.h_k\circ T^k)f
, f\rangle = {\lVert \sqrt{h_n.h_k \circ T^k}f \rVert}^2.$$\\

Now, for every $f\in L^2(\mu)$,

$$\langle (\lambda ^2 C_T^{k+1} {C_T^*}^{k+1}- C_T^k {C_T^*}^n C_T^n {C_T^*}^k) f, f \rangle ~ \geq 0.$$

$$\Longleftrightarrow \lambda^2 {\lVert \sqrt{h_{k+1}\circ T^{k+1}}f\rVert }^2 ~ \geq {\lVert \sqrt{h_n.h_k \circ T^k}f \rVert}^2.$$
Hence the proof.
\end{proof}

\begin{corollary} \label{Cor2.3}
Let $C_T \in B(L^2(\mu))$. Then:-\\
\begin{enumerate}
\item $C_T$ is posinormal if and only if $ \lambda {\lVert \sqrt{h}f\rVert } \geq {\lVert \sqrt{h \circ T}f\rVert}.$
\item$C_T^*$ is posinormal if and only if $ \lambda {\lVert \sqrt{h\circ T}f\rVert } ~ \geq {\lVert \sqrt{h}f \rVert}.$
\end{enumerate}

\end{corollary}

\begin{corollary}\label{cor2.4}
Let $C_T \in B(L^2(\mu))$. Then:-\\
\begin{enumerate}
\item $C_T$ is quasi posinormal if and only if $ \lambda {\lVert \sqrt{h_2}f\rVert } \geq {\lVert \sqrt{h}f\rVert}.$
\item $C_T^*$ is  quasi posinormal if and only if $ \lambda {\lVert \sqrt{h_2 \circ T^2}f\rVert } ~ \geq {\lVert \sqrt{h \circ T}f \rVert}.$
\end{enumerate}

\end{corollary}

\begin{corollary}\label{cor2.5}
Let $C_T \in B(L^2(\mu))$. Then:-\\
\begin{enumerate}
\item $C_T$ is $n$-power posinormal if and only if $ \lambda {\lVert \sqrt{h}f\rVert } \geq {\lVert \sqrt{h_n\circ T^n}f\rVert}.$
\item $C_T^*$ is $n$-power posinormal if and only if $ \lambda {\lVert \sqrt{h \circ T}f\rVert } ~ \geq {\lVert \sqrt{h_n}f \rVert}.$
\end{enumerate}

\end{corollary}

\begin{example}\label{Ex2.6}
Let $X=[0,1]$, $\Sigma$ be the Lebesgue measurable sets, $\mu$ be the Lebesgue measure, and $T(x)=1-x$.
For any Borel set $S\subseteq [0,1]$:
\begin{itemize}
  \item  $T^{-1}(S)=\{1-x:x\in S\}$.
  \item Since $T$ is an involution $(T^2=I)$, for $k\in \mathbb{N}$:
  $$T^{-k}(S)=\left\{
    \begin{array}{ll}
      S, & \hbox{if $k$ is even;} \\
     \{1-x:x\in S\}, & \hbox{if $k$ is odd.}
    \end{array}
  \right.$$
\end{itemize}
  \textbf{Radon-Nikodym Derivative} $h_k$:\\
  For any interval $(a,b)\subseteq [0,1]$:
  $$\mu\circ T^{-k}\big((a,b)\big)=\mu\big(T^{-k}((a,b))\big).$$
  \begin{enumerate}
    \item [Case I]: $k$ is even: $T^{-k}((a,b))=(a,b)$, so $\mu\circ T^{-k}\big((a,b)\big)=b-a$.
    \item [Case II]: $k$ is odd: $T^{-k}((a,b))=(1-b,1-a)$, so $\mu\circ T^{-k}\big((a,b)\big)=(1-a)-(1-b)=b-a$.
  \end{enumerate}
  In both cases, $\mu\circ T^{-k}(S)=\mu(S)$. Thus, the Radon-Nikodym Derivative:
  $$h_k(x)=1~(a.e.)~~\text{for all} k\in \mathbb{N}.$$
  \textbf{Posinormality of} $C_T$:\\
  By Corollary \ref{Cor2.3}(1),
  $C_T$ is posinormal if and only if  $\exists \lambda>0$ such that
  $$\lambda\big\|\sqrt{h}f\big\|\geq \big\|\sqrt{h}\|\circ T f\big\|.$$
  Substituting $h=1$:
  $$\lambda\norm{f}\geq \norm{f\circ T}.$$
As $T$ is measure-preserving $(\norm{f\circ T}=\norm{f})$, this reduces to $\lambda\norm{f}\geq \norm{f}$. Choosing $\lambda\geq 1$
satisfies the condition. Thus, $C_T$ is posinormal.\\
\textbf{Posinormality of} $C_T^*$:\\
By Corollary \ref{Cor2.3}(2), $C_T^*$ is posinormal if and only if
 $$\lambda\norm{\sqrt{h}\circ Tf}\geq \norm{\sqrt{h}f}.$$
  Substituting $h=1$:
  $$\lambda\norm{f\circ T}\geq \norm{f}.$$
  Again, $\norm{f\circ T}=\norm{f}$, so $\lambda\geq 1$ suffices. Thus $C_T^*$ is posinormal.\\
  \textbf{ $k$-quasi $n$-power posinormality}:\\
  By Theorem \ref{THM2.2} (i), $C_T$ is  $k$-quasi $n$-power posinormal if and only if
  $$\lambda \norm{\sqrt{h_{k+1}}f}\geq \norm{\sqrt{h_{k+1}}\circ T^{n-k}f}.$$
  Substituting $h_j=1$:
  $$\lambda\norm{f}\geq \norm{f\circ T^{n-k}}.$$
  Since $T^{n-k}$ is measure-preserving $(\norm{f\circ T^{n-k}}=\norm{f})$, the condition holds for $\lambda\geq 1$.
  Thus $C_T$ is $k$-quasi $n$-power posinormal for all $k,n\in \mathbb{N}$.
  Similarly,  $C_T^*$ is $k$-quasi $n$-power posinormal for all $k,n\in \mathbb{N}$ by Theorem \ref{THM2.2} (ii).
\end{example}




\section{$k$-quasi $n$-power Posinormal Weighted Composition Operators}
This section extends these results to weighted composition operators \(W_T = M_\pi C_T\). We characterize \(k\)-quasi \(n\)-power posinormality for \(W_T\) and its adjoint \(W_T^*\), assuming \(T^{-k}(\Sigma) = \Sigma\) and \(\pi \geq 0\). The criteria involve products of weights \(\pi_k\) and Radon-Nikodym derivatives \(h_k\), yielding norm inequalities analogous to those in Section 2. We denote the adjoint of $W_T$ by $W_T^*$. We denote $h_0=1$, $h_1=h$, $E_1=E$.

 \begin{theorem}\label{thm3.1} \cite{Campbell}
 Let $W_T$ be the weighted composition operator induced by $T$ and $\pi$ on $L^2(\mu)$, then the following statements hold for $f\in L^2(\mu)$ and $k\in \mathbb{N}$,
 \begin{enumerate}
 \item[(i)] $W_T^*f=h.E(\pi .f)\circ T^{-1}.$
 \item[(ii)] $W_T^k f =\pi_{k}.f\circ T^k.$
 \item[(iii)] $W_T^{*k}f=h_kE_k(\overline{\pi}_kf)\circ T^{-k}$.
 \item[(iv)]$W_T^*W_T(f)=hE(|\pi|^2)\circ T^{-1}(f).$
 \item[(v)] $W_TW_T^*(f)=\pi(h\circ T)E(\overline{\pi}f).$
 \item[(vi)] $W_T^{*k}W_T^k(f)=h_kE_k(|\pi_k |^2)\circ T^{-k}(f).$
 \item[(vii)] $W_T^kW_T^{*k}(f)=\pi_k (h_k\circ T^k)E_k(\overline{\pi}_k f),$
where $\pi_k=\pi(\pi \circ T)(\pi \circ T^2)\dots (\pi \circ T^{k-1}).$
\end{enumerate}

 \end{theorem}

   We assume that $\pi \geq 0$.
\begin{theorem}\label{thm3.2}  Let $W_T \in B(L^2(\mu ))$. Let $E$ be the projection from $L^2(\mu )$ onto $\overline {R(C_T)}$ and if $T^{-k}(\Sigma)=\Sigma$ for all $k$, then
\begin{enumerate}
\item [(i)] $W_T$ is $k$-quasi $n$-power posinormal if and only if, $$\lambda ^2h_{k+1}\pi^2_{k+1}\circ T^{-(k+1)}\geq \pi^2_k\pi^2_n (h_k\circ T^{-k})(h_n\circ T^n).$$\\
\item [(ii)] $W_T^*$ is $k$-quasi $n$-power posinormal  if and only if, $$\lambda^2 {\pi_k}^2(h_k\circ T^k)\geq {\pi_k}^2{\pi_n}^2(h_k\circ T^k)(h_n\circ T^{n+k}) .$$
\end{enumerate}
\end{theorem}

\begin{proof}
(i)  Now,
\begin{align*}
\lambda^2{W_T^*}^{k+1} W_T^{k+1} f &= \lambda^2h_{k+1}E_{k+1}(|\pi_{k+1}|^2)\circ T^{-(k+1)}f.\\
&= \lambda^2h_{k+1}\pi^2_{k+1}\circ T^{-(k+1)}f.
\end{align*}

\begin{align*}
{W_T^*}^kW_T^nW_T^{*n} W_T^k f &= {W_T^*}^kW_T^nW_T^{*n} [\pi_k. f\circ T^k]\\
&= {W_T^*}^k [\pi_n (h_n\circ T^n)E_n[\overline{\pi}_n\pi_k. f\circ T^k]]\\
&= h_k E_k\{\overline{\pi}_k[\pi_n (h_n\circ T^n)E_n[\overline{\pi}_n\pi_k. f\circ T^k]]\}\circ T^{-k}\\
&= h_k[E(\overline{\pi}_k\pi_n)E(h_n\circ T^n)E(\overline{\pi}_n\pi_k)E(f\circ T^k)]\circ T^{-k}\\
&= h_k[\overline{\pi}_k\pi_n (h_n\circ T^n) \overline{\pi}_n\pi_k (f\circ T^k)]\circ T^{-k}\\
&= \pi^2_k\pi^2_n(h_k\circ T^{-k})(h_n\circ T^n)f.
\end{align*}

Thus, for every $f\in L^2(\mu)$,

$$\langle (\lambda ^2{W_T^*}^{k+1} W_T^{k+1}- {W_T^*}^k W_T^n {W_T^*}^n W_T^k ) f, f \rangle ~ \geq 0.$$

$$\Longleftrightarrow \lambda ^2h_{k+1}\pi^2_{k+1}\circ T^{-(k+1)}\geq \pi^2_k\pi^2_n(h_k\circ T^{-k}) (h_n\circ T^n).$$
Hence the result.\\

(ii) 		Now,
\begin{align*}
\lambda^2 W_T^{k+1} {W_T^*}^{k+1} f &= \lambda^2 \pi_k(h_k\circ T^k)E_k(\overline{\pi}_kf)\\
&=\lambda^2 \pi_k(h_k\circ T^k)\pi_kf\\
&=\lambda^2 {\pi_k}^2(h_k\circ T^k)f.
\end{align*}

\begin{align*}
W_T^k{W_T^*}^nW_T^n{ W_T^*}^k f &= W_T^k{W_T^*}^nW_T^n[h_kE_k(\overline{\pi}_kf)\circ T^{-k}  ]\\
&= W_T^k [\pi_n(h_n\circ T^n)E_n(\overline{\pi}_n h_kE_k(\overline{\pi}_kf)\circ T^{-k})]\\
&= W_T^k [\pi_n(h_n\circ T^n)\pi_n h_k\pi_k(f\circ T^{-k})]\\
&=\pi_k.[(\pi_n(h_n\circ T^n)\pi_n h_k\pi_k(f\circ T^{-k}))\circ T^k]\\
&={\pi_k}^2{\pi_n}^2(h_k\circ T^k)(h_n\circ T^n\circ T^k)f
\end{align*}

Therefore, for every $f\in L^2(\mu)$,

$$\langle (\lambda ^2W_T^{k+1} {W_T^*}^{k+1}- W_T^k{W_T^*}^n W_T^n {W_T^*}^k) f, f \rangle ~ \geq 0.$$

$$\Longleftrightarrow \lambda^2 {\pi_k}^2(h_k\circ T^k)\geq {\pi_k}^2{\pi_n}^2(h_k\circ T^k)(h_n\circ T^{n+k}) .$$
\end{proof} 		
The following example demonstrates the applicability of Theorem \ref{thm3.2} for non-constant weights and identity transformations, with explicit verification of the key inequality.
\begin{example} Consider the weighted composition operator \(W_T\) on the Hilbert space \(L^2([0,1], \mu)\), where \(\mu\) is the Lebesgue measure. Define:
\begin{itemize}
  \item Transformation \(T\): The identity map \(T(x) = x\) for all \(x \in [0,1]\).
  \item Weight function \(\pi\): \(\pi(x) = x\) (non-constant and non-negative).
  \item Key Properties:
  \begin{enumerate}
    \item [i.] Radon-Nikodym derivatives: Since \(T\) is the identity, \(\mu \circ T^{-k} = \mu\) for all \(k \in \mathbb{N}\). Thus, \(h_k(x) = \frac{d(\mu \circ T^{-k})}{d\mu}(x) = 1\) for all \(k\) and almost every \(x\).
    \item [ii.] Iterated weight \(\pi_k\):
   \[
   \pi_k(x) = \pi(x) \cdot \pi(T(x)) \cdots \pi(T^{k-1}(x)) = x \cdot x \cdots x = x^k.
   \]
    \item [iii.] Adjoint and range condition: \(T^{-k}(\Sigma) = \Sigma\) holds trivially.
  \end{enumerate}
\end{itemize}
 Verification of Theorem \ref{thm3.2}(i):
\(W_T\) is \(k\)-quasi \(n\)-power posinormal if and only if:
\[
\lambda^2 h_{k+1} \pi_{k+1}^2 \circ T^{-(k+1)} \geq \pi_k^2 \pi_n^2 (h_k \circ T^{-k}) (h_n \circ T^n) \quad \text{a.e.}
\]
Substitute the values:
- \(h_{k+1}(x) = 1\), \(h_k(T^{-k}(x)) = 1\), \(h_n(T^n(x)) = 1\),
- \(\pi_{k+1}(T^{-(k+1)}(x)) = \pi_{k+1}(x) = x^{k+1}\),
- \(\pi_k(x) = x^k\), \(\pi_n(x) = x^n\).

The inequality simplifies to:
\[
\lambda^2 \cdot 1 \cdot (x^{k+1})^2 \geq (x^k)^2 (x^n)^2 \cdot 1 \cdot 1
\]
\[
\lambda^2 x^{2(k+1)} \geq x^{2k} x^{2n} = x^{2(k+n)}
\]
\[
\lambda^2 \geq x^{2(n-1)} \quad \text{a.e. on } [0,1].
\]
Conclusion:
Since \(x \in [0,1]\), the function \(x^{2(n-1)}\) is bounded by \(1\) for all \(n \geq 1\). Choosing \(\lambda \geq 1\) (e.g., \(\lambda = 1\)) satisfies the inequality. Thus, \(W_T\) is \(k\)-quasi \(n\)-power posinormal for all \(k, n \in \mathbb{N}\).

Operator Interpretation:
\(W_T f(x) = x \cdot f(x)\) is a multiplication operator. Its posinormality properties are confirmed by the above calculation, illustrating Theorem \ref{thm3.2}(i) concretely.
\end{example}

\section{$k$-Quasi $n$-Power Cauchy Dual of Lambert Conditional Operators}
This section examines the Cauchy dual \(\omega(T)\) of Lambert conditional operators \(T = M_w E M_u\), defined via the Moore-Penrose inverse. We derive conditions under which \(\omega(T)\) and its adjoint are \(k\)-quasi \(n\)-power posinormal. These conditions reduce to inequalities involving conditional expectations of \(u\) and \(w\) on a critical set \(K\) (Corollary 4.4). An example on \([0,1]\) demonstrates cases where \(\omega(T)\) is \(k\)-quasi \(n\)-power posinormal but not posinormal.

Here we consider the Cauchy dual of Lambert conditional operator $T=M_wEM_u$ introduced by Sohrabi in \cite{Mort} using the Moore-Penrose inverse. The Cauchy dual of $T\in B_C(\mathscr {H})$ is defined as,
 $ \omega (T)=T(T^*T)^ {\dagger} $.

 Let's recall few basic theorems.

 \begin{theorem}\label{thm4.1}
 \cite{Jabbar}Let $T=M_wEM_u\in B_C(L^2(\Sigma))$ with $u,w\geq 0$. Then $$T^{\dagger}=M_{\frac{\chi_K}{E(u^2)E(w^2}}T^*.$$
 \end{theorem}
We assume that $u, w\in L_+^0(\Sigma)$, $R=S(E(uw))$ and $K:=P\cap Q$, where $P=S(E(u))$ and $Q=S(E(w))$.
\begin{theorem}\label{thm4.2}
 \cite{Mort}Let $T=M_wEM_u\in B_C(L^2(\Sigma))$ with $u,w\geq 0$. Then we have
 \begin{enumerate}
 \item $\omega(T)=M_{\frac{\chi_K}{E(u^2)E(w^2)}}T$ and $\omega(T)^{\dagger}=T^*$.
 \item $\omega(T^*)=M_{\frac{\chi_K}{E(u^2)E(w^2)}}T^*=\omega(T)^*$.
 \end{enumerate}
 \end{theorem}

 \begin{theorem}\label{THM1}
Let $T=M_wEM_u\in B_C(L^2(\Sigma))$. Then the following statements hold, for some $\lambda >0$:-

\begin{enumerate}
\item[(i)]  If $\omega(T)$ is  posinormal, then $\lambda^2 E(uw)^2\geq E(u^2)E(w^2)$ on $K$.

\item[(ii)] If $\omega(T)$ is $n$-power posinormal, then $\lambda^2 E(u^2)^{2n-3}E(w^2)^{2n-3}\geq E(uw)^{2n-4}~\text{on}~ K.$

\item[(iii)] $\omega(T)$ is $k$-quasi $n$-power posinormal if and only if
\[
\lambda^2E(u^2)^{2n-3}E(w^2)^{2n-3} \geq E(uw)^{2n-4}~\text{on}~K.
 \]
  \item[(iv)]$\omega(T)^*$ is $k$-quasi $n$-power posinormal if and only if $\omega(T)$ is $k$-quasi $n$-power posinormal.

\end{enumerate}
\end{theorem}

\begin{proof}

(i) Suppose $\omega(T)$ is  posinormal, then $\lambda^2 \omega(T)^* \omega(T)- \omega(T)\omega(T)^* \geq 0$, for some $\lambda>0$.

For $f\in L^2(\Sigma)$, we have, $$\omega(T)^* \omega (T)f=\frac{\chi_K u}{E(u^2)^2E(w^2)}E(uf).$$
and
$$\omega(T) \omega (T)^*f=\frac{\chi_K w}{E(u^2)E(w^2)^2}E(wf).$$

Now, since $\omega(T)$ is posinormal, for all $f\in L^2(\Sigma)$,
\begin{equation}\label{intK2}
\int_{K}^{} \left( \lambda^2 \dfrac{uE(uf)}{E(u^2)^2E(w^2)}
-\dfrac{wE(wf)}{E(u^2)E(w^2)^2}\right) \overline{f}d\mu\geq0.
\end{equation}

Since $\mathcal{A}$ is $\sigma$-finite, there exists $\{A_m\}\subseteq \mathcal{A}$ such that $X=\cup A_m$, $A_m\subseteq A_{m+1}$
with $A_m\in \mathcal{A}$ with $0<\mu(A_m)<\infty$.

Substituting $f_m=w\sqrt{E(u^2)}\chi_{A_m}$ in \ref{intK2}, we get,
$$\int_{K\cap A_m}^{} \left( \lambda^2 \dfrac{uwE(uw)}{E(u^2)E(w^2)}
-\dfrac{w^2}{E(w^2)}\right) d\mu\geq0.$$
Hence, we obtain,  $\lambda^2 uwE(uw)\geq w^2E(u^2)$ on $K$.
Now, taking $E$ on both sides, we get,
$\lambda^2 E(uw)^2\geq E(u^2)E(w^2)$ on $K$.

(ii) Suppose $\omega(T)$ is  $n$-power posinormal, then $$\lambda^2 \omega(T)^* \omega(T)- \omega(T)^n\omega(T)^{*n} \geq 0 ~\text{for some}~\lambda>0 .$$

Now, for $f\in L^2(\Sigma)$, we have,
$$\omega(T)^* \omega (T)f=\frac{\chi_K u}{E(u^2)^2E(w^2)}E(uf).$$
and
$$\omega(T)^n \omega (T)^{*n}f=\frac{\chi_K wE(uw)^{2n-2}}{E(u^2)^{2n-1}E(w^2)^{2n}}E(wf).$$
Therefore, since $\omega(T)$ is $n$-power posinormal, for all $f\in L^2(\Sigma)$,
\begin{equation}\label{intK1}
\int_{K}^{} \left( \lambda^2 \dfrac{uE(uf)}{E(u^2)^2E(w^2)}
-\dfrac{wE(uw)^{2n-2}E(wf)}{E(u^2)^{2n-1}E(w^2)^{2n}}\right) \overline{f}d\mu\geq0.
\end{equation}
Since $\mathcal{A}$ is $\sigma$-finite, there exists $\{A_m\}\subseteq \mathcal{A}$ such that $X=\cup A_m$, $A_n\subseteq A_{m+1}$
with $A_m\in \mathcal{A}$ with $0<\mu(A_m)<\infty$
Substituting $f_m=w\sqrt{E(u^2)}\chi_{A_m}$ in \ref{intK1}, we get,
$$\int_{K\cap A_m}^{} \left( \lambda^2 \dfrac{uwE(uw)E(u^2)}{E(u^2)^2E(w^2)}
-\dfrac{w^2E(uw)^{2n-2}E(w^2)E(u^2)}{E(u^2)^{2n-1}E(w^2)^{2n}}\right) \overline{f}d\mu\geq0.$$
Thus, we get, $\lambda^2 u E(u^2)^{2n-3}E(w^2)^{2n-2}\geq wE(uw)^{2n-3}$ on $K$.

Now multiplying both sides by $w$ and applying $E$, we obtain,
$$\lambda^2 E(u^2)^{2n-3}E(w^2)^{2n-3}\geq E(uw)^{2n-4}~\text{on}~ K.$$

(iii) By definition, $\omega(T)$ is $k$-quasi $n$-power posinormal if and only if $$\lambda^2 \omega(T)^{*(k+1)}\omega(T)^{k+1}- \omega(T)^{*k}\omega(T)^n\omega (T)^{*n}\omega(T)^k \geq 0~\text{for some}~ \lambda>0.$$
For $f\in L^2(\Sigma)$, we have, $$\omega (T)f=M_{\frac{\chi_K}{E(u^2)E(w^2)}}Tf.$$
Therefore, for $n \in \mathbb {N}$ we get,
\begin{align*}
\omega(T)^n f&=\frac{\chi_Kw[E(uw)]^{n-1}}{(E(u^2))^n(E(w^2))^n}E(uf).
\end{align*}

Now,
$$\omega(T)^* f=M_{\frac{\chi_K}{E(u^2)E(w^2}}T^*f.$$

Similarly, for $n \in \mathbb {N}$ we get,
\begin{align*}
\omega(T)^{*n} f&=\frac{\chi_Ku[E(uw)]^{n-1}}{E(u^2)^nE(w^2)^n}E(wf).
\end{align*}

Now,
\begin{align*}
\omega(T)^{*k+1} \omega(T)^{k+1}f
&=\frac{\chi_Ku[E(uw)]^{2k}E(uf)}{E(u^2)^{2k+2}E(w^2)^{2k+1}}.
\end{align*}

By similar computation we get,
\begin{align*}
\omega(T)^{*k} \omega(T)^n\omega(T)^{*n} \omega(T)^kf
&=\frac{\chi_Ku[E(uw)]^{2k+2n-4}E(uf)}{E(u^2)^{2k+2n-1}E(w^2)^{2k+2n-2}}.
\end{align*}

Then we obtain,\\

$\langle (\lambda^2 \omega(T)^{*(k+1)} \omega(T)^{k+1}- \omega(T)^{*k}\omega(T)^n\omega (T)^{*n}\omega(T)^k)f, f\rangle$ \\\\

$=\bigint_{X}^{} \left( \lambda^2 \dfrac{\chi_Ku[E(uw)]^{2k}E(uf)}{E(u^2)^{2k+2}E(w^2)^{2k+1}}-\dfrac{\chi_Ku[E(uw)]^{2k+2n-4}E(uf)}{E(u^2)^{2k+2n-1}E(w^2)^{2k+2n-2}}\right) \overline{f}d\mu$\\

$=\bigint_{K}^{} \left(\lambda^2  \dfrac{[E(uw)]^{2k}}{E(u^2)^{2k+2}E(w^2)^{2k+1}}-\dfrac{[E(uw)]^{2k+2n-4}}{E(u^2)^{2k+2n-1}E(w^2)^{2k+2n-2}}\right)|E(uf)|^2d\mu$\\\\

$=\bigint _{K}^{} \dfrac{E(uw)^{2k}}{E(u^2)^{2k+2}E(w^2)^{2k+1}}\left( \lambda^2 -\dfrac{E(uw)^{2n-4}}{E(u^2)^{2n-3}E(w^2)^{2n-3}}\right)|E(uf)|^2d\mu$.\\

This implies that if $\lambda^2 E(u^2)^{2n-3}E(w^2)^{2n-3} \geq E(uw)^{2n-4}$ on $K$, then $\omega(T)$ is $k$-quasi $n$-power posinormal operator for every $n, ~k\in \mathbb{N}$.

Conversely, assume that $\omega(T)$ is a $k$-quasi $n$-power posinormal operator. Then for all $f\in L^2(\Sigma)$, we have
$$\langle (\lambda^2 \omega(T)^{*(k+1)} \omega(T)^{k+1}- \omega(T)^{*k}\omega(T)^n\omega (T)^{*n}\omega(T)^k)f, f\rangle \geq 0.$$

Let $B\in \mathscr{A}$, with $B\subseteq K$ and $0<\mu(B)<\infty$. Taking $f=\chi_B$, we have

$$\bigint_{K}^{} \left( \lambda^2 \dfrac{[E(uw)]^{2k}}{E(u^2)^{2k+2}E(w^2)^{2k+1}}-\dfrac{[E(uw)]^{2k+2n-4}}{E(u^2)^{2k+2n-1}E(w^2)^{2k+2n-2}}\right)|E(u\chi_B)|^2d\mu\geq 0.$$

Therefore,\\
$\bigint_{B}^{} \left( \lambda^2 \dfrac{[E(uw)]^{2k}}{E(u^2)^{2k+2}E(w^2)^{2k+1}}-\dfrac{[E(uw)]^{2k+2n-4}}{E(u^2)^{2k+2n-1}E(w^2)^{2k+2n-2}}\right)E(u^2)d\mu\geq 0.$\\\\
Since $B \in \mathscr{A}$ is arbitrary, we get $\lambda^2 E(u^2)^{2n-3}E(w^2)^{2n-3} \geq E(uw)^{2n-4}$  on $K$. Hence the proof.

(iv) By similar computation as the above it can be shown that,
 $\omega(T)^*$ is $k$-quasi $n$-power posinormal if and only if
 $\lambda^2E(u^2)^{2n-3}E(w^2)^{2n-3} \geq E(uw)^{2n-4}~\text{on}~K.$
Then by Theorem \ref{THM1} (3), the result follows.

\end{proof}

\begin{corollary}\label{cor4.1}
Let $T=M_wEM_u\in B_C(L^2(\Sigma))$. Then,
\begin{enumerate}
 \item $\omega(T)$ is $n$-power posinormal if and only if
$$\lambda^2E(u^2)^{2n-3}E(w^2)^{2n-3} \geq E(uw)^{2n-4}~\text{on}~ K.$$
\item $\omega(T)$ is posinormal  if and only if
$\lambda^2E(uw)^2\geq E(u^2)E(w^2)$ on $K$.
\end{enumerate}
\end{corollary}

\begin{example}\label{ex.4.2}
Let $X=[0,1]$, $d\mu=dx$, $\Sigma$ be the Lebesgue measurable sets and let $\mathscr{A}=\{\phi,X\}$. Therefore, $Tf(x)=w(x)E(uf)(x)=w(x)\int_{0}^{1}u(x)f(x)dx$. Taking $u(x)=\frac{x}{2\sqrt{2}}$ and $w(x)=5x^2+3$, we get $E(u^2)=\frac{1}{24}$, $E(w^2)=24$ and $E(uw)=\frac{11}{8\sqrt{2}}$. Therefore, for $n=2$, $\lambda^2 E(u^2)E(w^2)\geq 1$ if and only if $\lambda^2\geq1$.Thus, if $\lambda^2\geq1$, $\omega(T)$ is $k$-quasi $2$-power posinormal for all $k\in \mathbb{N}$. In general, for $\lambda^2\geq (0.98)^{2n-4}$, then $\omega(T)$ is $k$-quasi $n$-power posinormal for $k,n \in \mathbb{N}$.\\
Now, $\omega(T)$ posinormal if and only if $\lambda^2\geq1.06$.
Therefore,  there exists $\omega(T)$'s which are $k$-quasi $n$-power posinormal, but not posinormal (for example, take  $\lambda^2=0.9$ and $n=4.$)

\end{example}

\section{$k$-Quasi $n$-Power Posinormal Weighted Shifts on Directed Trees}
This section concludes with an application to weighted shifts on rooted directed trees. After reviewing graph-theoretic preliminaries, we characterize \(k\)-quasi \(n\)-power posinormality for shifts \(S_\mu\) and provide a detailed example on a tree with one root and two branches . The analysis reveals explicit norm conditions dependent on the tree’s weight structure.

A pair $\mathscr{T}=(V,E)$ is a \textit{directed graph} if $V$ is a nonempty set and $E$ is a subset of $V\times V \backslash \{(v,v): v\in V\}$. An element of $V$ is called a \textit{vertex} of $\mathscr{T}$ and an element $E$ is called an \textit{edge} of $\mathscr{T}$.  A directed graph $\mathscr{T}$ is \textit{connected }if for any two distinct vertices $u$ and $v$ of $\mathscr{T}$, there exists a finite sequence $v_1, v_2,\cdots, v_n$ of vertices of $\mathscr{T}$ $(n\geq 2)$ with $u=v_1$, $v=v_n$ and $(v_j,v_{j+1})\in E$ or  $(v_{j+1},v_j)\in E$ for all $1 \leq j\leq n-1$. A  finite sequence $\{v_i\}_{i=1}^{n}$ of distinct vertices is said to be a circuit of $\mathscr{T}$ if $n\geq2$, for $1\leq i \leq n-1$ $(v_i,v_{i+1})\in E$ and $(v_n,v_1)\in E$.  For $W\subset V$, define $\text{Chi}(W)=\bigcup_{u\in W}\{v\in V:(u,v)\in E\}$. Inductively, for $n\in \mathbb{Z}_+$, $\text{Chi}^{<n>}(W)=W$ if $n=0$ and $\text{Chi}^{<n>}(W)=\text{Chi}(\text{Chi}^{<n-1>}W)$, if $n\geq1$. Given $v\in V$,  $\text{Chi}(v)= \text{Chi}(\{v\})$ and  $\text{Chi}^{<n>}(v)= \text{Chi}^{<n>}(\{v\})$.
 A member of $\text{Chi}(v)$ is called a \textit{child} of $v$. For  $u\in V$ if there exists a unique vertex $v\in V$ such that $(v,u)\in E$, then $v$ is said to be a \textit{parent} of
$u$ and is denoted as $\text{par}(u)=v$. A vertex $v$ of $\mathscr{T}$ is called a \textit{root} of $\mathscr{T}$ or $v\in Root(\mathscr{T})$, if there is no vertex $u$ of $\mathscr{T}$ such that $(u,v)\in E$. If $Root(\mathscr{T})$ is a singleton set, then the unique element is denoted by root. We write,
$V^0:=V \backslash Root(\mathscr{T})$. 
A directed graph $\mathscr{T}=(V,E)$ is called a \textit{directed tree} if $\mathscr{T}$ is connected,
$\mathscr{T}$ has no circuits, and each vertex $v\in V^0$ has a parent. Any directed tree has atmost one root \cite{Jablon}.

Let $\ell^2(V)$ denotes the Hilbert space of square summable complex functions on $V$ equipped with the standard inner product. The set $\{e_u\}_{u\in V}$ is an orthonormal basis of $\ell^2(V)$, where $e_u \in \ell^2(V)$ is the indicator function of $\{u\}$. For a system $\mu=\{\mu_v\}_{v\in V^0}$
of non-negative real numbers, define the \textit{weighted shift operator} $S_{\mu}$ on $\mathscr{T}$ with weights $\bold{\mu}$ by
$$S_{\mu}f:=\Lambda_{\mathscr{T}}f, f\in \mathscr{D}(S_{\mu}),$$
where, $$\mathscr{D}(S_{\mu}):=\{f \in \ell^2(V): \Lambda_{\mathscr{T}}f \in \ell^2(V)\}$$ and for $f\in \ell^2(V)$, $\Lambda _{\mathscr{T}}$ is the mapping given by

$$(\Lambda _{\mathscr{T}}f)(v):=
\begin{cases}
	\mu_v. f(\text{par}(v))& \text{if } v\in V^0,\\
	0 & \text{if v is a root of}~ \mathscr{T}.
\end{cases}$$

If $S_\mu$ is a densely defined weighted shift on the directed tree $\mathscr{T}$ with weights $\mu=\{ \mu_v\}_{v\in V^0}$, then
 $e_u\in \mathscr{D}(S_\mu^*)$ and $$S_{\mu}^*e_u=
\begin{cases}
\overline{\mu_u}e_{\text{par}(u)}&;~ u\in V^0\\
0 &; ~ \text{if u is a root of}~ \mathscr{T}
\end{cases}$$ \cite{Jablon}. Let $S_{\mu}$ be a weighted shift on a directed tree $\mathscr{T}$ with weights $\mu=\{\mu_v\}_{v\in V^0}$ and let $\sup\limits_{u\in V} \text{card}(\text{Chi}(u))<\infty$. Then $S_{\mu}\in B(\ell^2(V))$ if and only if $\sup\limits_{v\in V^0}|\mu_v|<\infty$  \cite{Jablon}. We refer \cite{Jablon} for more information regarding a weighted shift on a directed tree.

\begin{theorem}\label{thm5.4}
Let $S_{\mu}\in B(\ell^2(V))$ be a non-zero weighted shift on a directed tree $\mathscr{T}$ with weights $\mu=\{\mu_v\}_{v\in V^0}$. Then $S_{\mu}$ is $k$-quasi $n$-power posinormal operator if and only if $\lVert S_{\mu}^{*n}S_{\mu}^k f \rVert^2 \leq \lambda^2 \lVert S_{\mu}^{k+1}f\rVert^2$, for all $f\in \ell^2(V)$ and for some $\lambda >0$.
\end{theorem}

\begin{proof}
$S_{\mu}$ is $k$-quasi $n$-power posinormal operator if and only if
$$S_{\mu}^{*k}[\lambda ^2 S_{\mu}^*S_{\mu}-S_{\mu}^n S_{\mu}^{*n}]S_{\mu}^k\geq 0.$$
$$\Longleftrightarrow \langle (\lambda ^2 S_{\mu}^{*k+1}S_{\mu}^{k+1}-S_{\mu}^{*k}S_{\mu}^n S_{\mu}^{*n}S_{\mu}^k)f,f\rangle\geq0,~ \text{for all} ~f\in \ell^2(V). $$
$$\Longleftrightarrow \lVert S_{\mu}^{*n}S_{\mu}^k f \rVert^2 \leq \lambda^2 \lVert S_{\mu}^{k+1}f\rVert^2 ~ \text{for all} ~f\in \ell^2(V).$$

\end{proof}

\begin{example}\label{ex5.5}
Let $\mathscr{T}=(V,E)$ be a directed tree with one root and $V^0=\{v_{ij};i=1,2 ~\text{and}~ j=1,2, \cdots\}$ such that $\text{card}(\text{Chi}(root))=2$ and $\text{card}(\text{Chi}(v))=1$ for all $v\in V^0$.
Let $S_{\mu}$ be the weighted shift on $\mathscr{T}$ with weights given by,
$$\mu_{v_{11}}=1, ~\mu_{v_{1j}}=2 ~\text{for all} ~j\geq2$$
$$\text{and}~\mu_{v_{22}}=\frac{1}{2}, ~\mu_{v_{2j}}=2 ~\text{for all} ~j \neq 2.$$

By \ref{PROP1}, $S_{\mu}\in B(\ell^2(V))$.

Now, $S_{\mu}e_{root}=e_{v_{11}}+e_{v_{21}}$,  $S_{\mu}(e_{v_{1j}})=2  e_{v_{1j+1}} $, for $j\geq1$.

$S_{\mu}(e_{v_{21}})=\frac{1}{2}  e_{v_{22}} $, $S_{\mu}(e_{v_{2j}})=  e_{v_{2j+1}} $, for $j\geq2$.\\
Now,
$$S_{\mu}^*e_u=
\begin{cases}
\overline{\mu_u}e_{\text{par}(u)}&;~ u\in V^0,\\
0 &; ~ u=\text{root}.
\end{cases}$$

Therefore, $S_{\mu}^*(e_\text{root})=0$,

$$S_{\mu}^*e_{v_{1j}}=
\begin{cases}
e_{\text{root}} &; ~ j=1,\\
2e_{v_{1j-1}} &;~j\geq 2.
\end{cases}$$

and
$$S_{\mu}^*e_{v_{2j}}=
\begin{cases}
e_{\text{root}} &; ~ j=1,\\
\frac{1}{2}e_{v_{21}} &; j=2,\\
e_{v_{2j-1}} &;~j> 2.
\end{cases}$$

Let $f\in \ell^2(V)$ and $f=\alpha_0e_{root}+\sum\limits_{j=1}^{\infty}\alpha_{1j}e_{v_{1j}}+\sum\limits_{j=1}^{\infty}\alpha_{2j}e_{v_{2j}}$. Then,

$S_{\mu}f=\alpha_0e_{v_{11}}+\alpha_0e_{v_{21}}+\sum\limits_{j=1}^{\infty}\alpha_{1j}e_{v_{1j+1}}+\frac{1}{2} \alpha_{21} e_{v_{22}}+\sum\limits_{j=1}^{\infty}\alpha_{2j}e_{v_{2j+1}}$.

In general,

$S_{\mu}^kf=2^{k-1}\alpha_0e_{v_{1k}}+\frac{1}{2}\alpha_0e_{v_{2k}}+2^k\sum\limits_{j=2}^{\infty}\alpha_{1j}e_{v_{1j+k}}+\frac{1}{2} \alpha_{21} e_{v_{2k+1}}+\sum\limits_{j=2}^{\infty}\alpha_{2j}e_{v_{2j+k}}$.

Therefore,

$\lVert S_{\mu}^{k+1}f\rVert^2=2^{2k}|\alpha_0|^2+\frac{1}{4}|\alpha_0|^2+2^{2k+2}\sum\limits_{j=2}^{\infty}|\alpha_{1j}|^2+\frac{1}{4}|\alpha_{21}|^2+\sum\limits_{j=2}^{\infty}|\alpha_{2j}|^2$.

By computations, we get, for $k-n>1$,
$$S_{\mu}^{*n}S_{\mu}^kf=2^{k+n-1}\alpha_0e_{v_{1k-n}}+\frac{1}{2}\alpha_0e_{v_{2k-n}}+2^{k+n}\sum\limits_{j=2}^{\infty}\alpha_{1j}e_{v_{1j+k-n}}+\frac{1}{2} \alpha_{21} e_{v_{2k-n+1}}+\sum\limits_{j=2}^{\infty}\alpha_{2j}e_{v_{2j+k-n}}.$$

Therefore, for $k-n>1$,
$$\lVert S_{\mu}^{*n}S_{\mu}^kf\rVert^2=2^{2k+2n-2}|\alpha_0|^2+\frac{1}{4}|\alpha_0|^2+2^{2k+2n}\sum\limits_{j=2}^{\infty}|\alpha_{1j}|^2++\sum\limits_{j=2}^{\infty}|\alpha_{2j}|^2.$$

Taking, $$y_1=2^{2k}|\alpha_0|^2, y_2=\frac{1}{4}|\alpha_0|^2, y_3= 2^{2k+2}\sum\limits_{j=2}^{\infty}|\alpha_{1j}|^2, y_4=\frac{1}{4}|\alpha_{21}|^2~ \text{and}~ y_5=\sum\limits_{j=2}^{\infty}|\alpha_{2j}|^2,$$

we get, for $k>n+1$ $S_{\mu}$ is $k$-quasi $n$-power posinormal if and only if $$\lambda^2\geq \dfrac{2^{2n-2}y_1+y_2+2^{2n}y_3+y_4+y_5}{y_1+y_2+y_3+y_4+y_5}.$$
Since $y_i\geq 0$, for all $i$, we can see that, $2^{2n}$ is a upper bound for $\dfrac{2^{2n-2}y_1+y_2+2^{2n}y_3+y_4+y_5}{y_1+y_2+y_3+y_4+y_5}.$
Therefore, taking $\lambda\geq2^{n}$, $S_{\mu}$ is $k$-quasi $n$-power posinormal for $k>n+1$.
\end{example}
\section{Conclusion and Future Work}
This paper has systematically investigated the class of \(k\)-quasi \(n\)-power posinormal operators across diverse settings. For composition operators \(C_T\) on \(L^2(\mu)\), we established necessary and sufficient conditions for \(k\)-quasi \(n\)-power posinormality, expressed via Radon-Nikodym derivatives \(h_k\) and the transformation \(T\). Analogous characterizations were derived for weighted composition operators \(W_T\), involving products of weights \(\pi_k\) and shifts of \(h_k\). For Lambert conditional operators \(T = M_w E M_u\), we characterized when their Cauchy duals \(\omega(T)\)— constructed via the Moore-Penrose inverse— are \(k\)-quasi \(n\)-power posinormal, yielding the unified condition \(\lambda^2 E(u^2)^{2n-3} E(w^2)^{2n-3} \geq E(uw)^{2n-4}\) on \(K\). Finally, we extended these results to weighted shifts \(S_\mu\) on rooted directed trees\ and provided explicit examples validating the theory. These contributions generalize classical posinormal operator theory and bridge measure theory, conditional expectations and operator theory.

Several directions emerge naturally from this work. First, relaxing the assumption \(T^{-k}(\Sigma) = \Sigma\) would extend the applicability of the results to non-measure-preserving transformations. Second, exploring the class beyond \(L^2\)-spaces— such as in \(L^p\) (\(p \neq 2\)), non-commutative \(L^2\)-spaces of von Neumann algebras, or Hardy/Bergman spaces— could reveal deeper connections to harmonic analysis. Third, investigating applications in frame theory (e.g., optimal frame bounds via Cauchy duals), quantum information (e.g., posinormal quantum channels), or network dynamics (e.g., stability of flows modelled by weighted shifts on trees) would demonstrate the practical utility of these operators. Additionally, extending the framework to broader operator classes (e.g., Toeplitz operators, matrix-valued weighted shifts) or linking it to Aluthge transformations and subnormality remains open. Computational methods to verify posinormality for high-dimensional operators could also be developed.

\bibliographystyle{amsplain}
\bibliography{references}

\end{document}